\documentclass[11pt,reqno]{amsart}
\usepackage[utf8]{inputenc}
\usepackage[T1]{fontenc}
\usepackage{amsrefs}
\usepackage{mathrsfs}
\usepackage{tikz-cd}

\numberwithin{equation}{section}

\theoremstyle{plain}
\newtheorem{theorem}{Theorem}[section]

\newtheorem{lemma}[theorem]{Lemma}
\newtheorem{corollary}[theorem]{Corollary}
\theoremstyle{definition}

\theoremstyle{remark}

\newtheorem{example}[theorem]{Example}

\newcommand{\C}{\mathbb{C}}

\newcommand{\Z}{\mathbb{Z}}

\renewcommand{\P}{\mathbb{P}}

\newcommand{\isheaf}{\mathscr{I}}

\newcommand{\curr}{\mathscr{C}}

\newcommand{\holo}{\mathscr{O}}
\newcommand{\opens}{\mathscr{U}}
\newcommand{\red}{\textnormal{red}}

\renewcommand{\dbar}{\bar{\partial}}

\DeclareMathOperator{\codim}{codim}
\DeclareMathOperator{\End}{End}

\DeclareMathOperator{\Hom}{Hom}

\DeclareMathOperator{\id}{id}
\DeclareMathOperator{\im}{im}

\title{Polynomial interpolation and residue currents}
\author{Jimmy Johansson}
\address{Department of Mathematical Sciences,
Chalmers University of Technology and the University of Gothenburg,
Gothenburg SE-412 96, Sweden}
\email{jimjoh@chalmers.se}
\date{January 20, 2021}

\begin{document}
\begin{abstract}
We show that a global holomorphic section of $\holo(d)$ restricted to a closed complex subspace $X \subset \P^n$ has an interpolant if and only if it satisfies a set of moment conditions that involves a residue current associated with a locally free resolution of $\holo_X$. When $X$ is a finite set of points in $\C^n \subset \P^n$ this can be interpreted as a set of linear conditions that a function on $X$ has to satisfy in order to have a polynomial interpolant of degree at most $d$.
\end{abstract}
\maketitle
\section{Introduction}
\label{section:introduction}
Let $i: X \hookrightarrow \C^n$ be a subvariety or complex subspace whose underlying space, $X_\red$, is a finite set of points $\{ p_0, \dots, p_r \} \subset \C^n$. Let $g$ be a holomorphic function on $X$, i.e., a global holomorphic section of $\holo_X$, and let $G \in \C[\zeta_1,\dots,\zeta_n]$ be a polynomial. We say that $G$ interpolates $g$ if the pull-back of $G$ to $X$ equals $g$, i.e., $i^*G = g$.

If $X$ is reduced, then a holomorphic function $g$ on $X$ is just a function from $X_\red = \{ p_0, \dots, p_r \}$ to $\C$, and we have that $G$ interpolates $g$ if $G(p_j) = g(p_j)$ for each $j = 0, \dots, r$. In the univariate case this is referred to as Lagrange interpolation. If $X$ is not reduced, then at each point $p_j$, $G$ also has to satisfy some conditions on its derivatives. In the univariate case this is referred to as Hermite interpolation, see Example~\ref{ex:hermite}.

The motivating question for this note is the following. What are the necessary and sufficient conditions on $g$ for the existence of an interpolant of degree at most $d$?

Let $A_X$ denote the vector space of holomorphic functions on $X$, i.e., $A_X = H^0(\C^n,\holo_X)$. Since the set of holomorphic functions on $X$ that have an interpolant of degree at most $d$ is a linear subspace of $A_X$, we have that a function $g \in A_X$ has an interpolant of degree at most $d$ if and only if it satisfies a finite set of linear conditions. In this note we will show how these linear conditions can be explicitly realized as a set of moment conditions that involves a so-called residue current associated with a locally free resolution of $\holo_X$.

Recall that since $X_\red$ is a finite set of points, $X$ can be viewed as a closed complex subspace of $\P^n$, and we have that polynomials of degree at most $d$ on $\C^n$ naturally correspond to global holomorphic sections of the line bundle $\holo(d) \rightarrow \P^n$ via $d$-homogenization. This motivates the following more general notion of interpolation that we shall consider in this note. Let $i: X \hookrightarrow \P^n$ be a closed complex subspace of arbitrary dimension. Let $\Phi$ and $\varphi$ be global holomorphic sections of $\holo(d)$ and $\holo_X(d) = i^*\holo(d)$, respectively. We say that $\Phi$ interpolates $\varphi$ if $i^*\Phi = \varphi$.

From a minimal graded free resolution of the homogeneous coordinate ring of $X$, $S_X$, we obtain a locally free resolution of $\holo_X$ of the form
\begin{equation}
\label{eq:locally-free-res-pn}
	0 \longrightarrow
	\holo(E_n) \overset{f_n}{\longrightarrow} \dots
	\overset{f_2}{\longrightarrow} \holo(E_1) \overset{f_1}{\longrightarrow} \holo_{\P^n}
	\longrightarrow \holo_X \longrightarrow 0
\end{equation}
where $E_k = \bigoplus_\ell \holo(-\ell)^{\beta_{k,\ell}}$, see \cite{eisenbud} and \cite{aw}*{Section~6}. The $\beta_{k,\ell}$ are referred to as the \emph{graded Betti numbers} of $S_X$. We equip the $E_k$ with the natural Hermitian metrics. In \cite{aw}, Andersson and Wulcan showed that with (\ref{eq:locally-free-res-pn}), one can associate a residue current $R$ that generalizes the classical Coleff--Herrera product \cite{ch}, see Section~\ref{section:residue-currents}. It can be written as $R = \sum_{k,\ell} R_{k,\ell}$, where each $R_{k,\ell}$ is an $\holo(-\ell)^{\beta_{k,\ell}}$-valued $(0,k)$-current. In \cite{aw2}, the same authors proved a result which as a special case gives a cohomological condition in terms of the current $R$ for when $\Phi$ interpolates $\varphi$. In Section~\ref{section:interpolation} we will show that in our setting this condition amounts to the following set of moment conditions.
\begin{theorem}
\label{thm:moment-conditions}
Let $X \subset \P^n$ be a closed complex subspace, and let $R$ be the residue current associated with (\ref{eq:locally-free-res-pn}). Moreover, let $\omega$ be a nonvanishing holomorphic $\holo(n+1)$-valued $n$-form. A global holomorphic section $\varphi$ of $\holo_X(d)$ has an interpolant if and only if for each $\ell$ it holds that
\begin{equation}
\label{eq:integral-condition}
	\int_{\P^n} R_{n,\ell} \varphi \wedge h \omega = 0
\end{equation}
for all global holomorphic sections $h$ of $\holo(\ell-d-n-1)$.
\end{theorem}

Recall that the \emph{interpolation degree} of $X$ is defined as
\[
	\inf \{ d: \text{all global holomorphic sections of $\holo_X(d)$ has an interpolant} \}.
\]
In particular, if $X_\red$ is a finite set of points in $\C^n$, then the interpolation degree of $X$ is the smallest number $d$ such that any $g \in A_X$ has an interpolant of degree at most $d$. Define
\begin{equation}
\label{eq:tk-sx}
	t_k(S_X) = \sup \{ \ell: \beta_{k,\ell} \neq 0 \}.
\end{equation}
(We use the convention that the supremum of the empty set is $-\infty$.) As a consequence of Theorem~\ref{thm:moment-conditions} we get the following bound of the interpolation degree.
\begin{corollary}
\label{cor:interpolation-degree}
Let $X \subset \P^n$ be a closed complex subspace with homogeneous coordinate ring $S_X$. The interpolation degree of $X$ is less than or equal to $t_n(S_X)-n$.
\end{corollary}
It can be shown by purely algebraic means that the interpolation degree of $X$ is in fact equal to $t_n(S_X)-n$, see e.g. \cite{johansson}*{Corollary~1.6}. If $X_\red$ consists of a finite set of points, then it can be shown that $t_n(S_X)-n$ is equal to the Castelnuovo--Mumford regularity of $S_X$, see \cite{eisenbud}*{Exercise 4E.5}, and the statement in this case is Theorem~4.1 in \cite{eisenbud}.

In Section~\ref{section:polynomial-interpolation} we will consider the case when $X_\red$ is a finite set of points in $\C^n$. In this case the corresponding versions of Theorem~\ref{thm:moment-conditions} and Corollary~\ref{cor:interpolation-degree} first appeared in \cite{standar}. We will consider some examples where we explicitly write down the conditions for when $g \in A_X$ has an interpolant of degree at most $d$. In particular, we will obtain the precise conditions for when the Hermite interpolation problem has a solution.
\section{Residue currents}
\label{section:residue-currents}
Let $f$ be a holomorphic function in an open set in $\C^n$. Let $\xi$ be a smooth $(n,n)$-form with compact support. In \cite{hl}, using Hironaka's desingularization theorem, Herrera and Lieberman proved that the limit
\begin{equation}
\label{eq:principal-value}
	\lim_{\epsilon \rightarrow 0}
	\int_{|f| > \epsilon} \frac{\xi}{f}
\end{equation}
exists. Thus (\ref{eq:principal-value}) defines a current known as the principal value current, which is denoted by $[1/f]$. The residue current $R^f$ of $f$ is the $(0,1)$-current $\dbar [1/f]$. It is easy to see that $R^f$ has its support on $V(f) = f^{-1}(0)$, and that it satisfies the following \emph{duality principle}: A holomorphic function $\Phi$ belongs to the ideal $(f)$ if and only if $R^f \Phi = 0$.
\begin{example}
Let $\zeta_0 \in \C$. We have that the action of $\dbar [1/(\zeta-\zeta_0)]$ on a test form $\xi(\zeta) d \zeta$ is given by
\begin{equation}
\label{eq:dbar-zk}
	\left\langle \dbar
	\left[ \frac{1}{\zeta-\zeta_0} \right],
	\xi(\zeta) d\zeta \right\rangle =
	2 \pi i \xi(\zeta_0).
\end{equation}
\qed
\end{example}

\subsection{Residue currents associated with generically exact complexes}
\label{section:aw-current}
We will now consider a generalization of the above construction due to Andersson and Wulcan. Consider a generically exact complex of Hermitian holomorphic vector bundles over a complex manifold $Y$ of dimension $n$,
\begin{equation}
\label{eq:bundle-complex}
	0 \longrightarrow
	E_n \overset{f_n}{\longrightarrow} \dots
	\overset{f_2}{\longrightarrow} E_1 \overset{f_1}{\longrightarrow} E_0
	\longrightarrow 0,
\end{equation}
i.e., a complex that is exact outside an analytic variety $Z \subset Y$ of positive codimension. The vector bundle $E = \bigoplus_k E_k$ has a natural superbundle structure, i.e., a $\Z_2$-grading, $E = E_+ \oplus E_-$, where $E_+ = \bigoplus_k E_{2k}$ and $E_- = \bigoplus_k E_{2k+1}$, which we shall refer to as the subspaces of even and odd elements, respectively. This induces a $\Z_2$-grading on the sheaf of $E$-valued currents $\curr(E)$; if $\omega \otimes \xi$ is an $E$-valued current, where $\omega$ is a current and $\xi$ is a smooth section of $E$, then the degree of $\omega \otimes \xi$ is the sum of the degree of $\xi$ and the current degree of $\omega$ modulo 2.

We say that an endomorphism on $E$ is even (resp.\ odd) if it preserves (resp.\ switches) the degree. If $\alpha$ is a smooth section of $\End E$, then it defines a map on $\curr(E)$ via
\[
	\alpha(\omega \otimes \xi) =
	(-1)^{(\deg \alpha)(\deg \omega)} \omega \otimes \alpha(\xi),
\]
where $\omega$ is a current and $\xi$ is a smooth section of $E$. In particular, the map $f = \sum_{k=1}^n f_k$ defines an odd map on $\curr(E)$. We define an odd map on $\curr(E)$, $\nabla = f-\dbar$, which, since $f$ and $\dbar$ anti-commute, satisfies $\nabla^2 = 0$. The map $\nabla$ extends to an odd map on $\curr(\End E)$ via Leibniz's rule,
\[
	\nabla(\alpha \xi) = (\nabla \alpha) \xi + (-1)^{\deg \alpha} \alpha \nabla \xi.
\]

In \cite{aw}, Andersson and Wulcan constructed $\End E$-valued currents
\[
	U = \sum_\ell U^\ell =
	\sum_\ell \sum_{k \geq \ell+1} U_k^\ell,
\]
and
\[
	R = \sum_\ell R^\ell =
	\sum_\ell \sum_{k \geq \ell+1} R_k^\ell,
\]
where $U_k^\ell$ and $R_k^\ell$ are $\Hom(E_\ell, E_k)$-valued currents of bidegree $(0, k-\ell-1)$ and $(0, k-\ell)$, respectively, which satisfy
\begin{equation}
\label{eq:nabla-U-R}
	\nabla U = \id_E - R, \quad \nabla R = 0.
\end{equation}
The current $R$ is referred to as the residue current associated with (\ref{eq:bundle-complex}) and it has its support on $Z$.

Suppose that the complex of locally free sheaves corresponding to (\ref{eq:bundle-complex}),
\begin{equation}
\label{eq:hermitian-res}
	0 \longrightarrow
	\holo(E_n) \overset{f_n}{\longrightarrow} \dots
	\overset{f_2}{\longrightarrow} \holo(E_1) \overset{f_1}{\longrightarrow} \holo(E_0),
\end{equation}
is exact. When the $E_k$ are equipped with Hermitian metrics, we shall refer to (\ref{eq:hermitian-res}) as a \emph{Hermitian resolution} of the sheaf $\holo(E_0) / \im f_1$. In this case it holds that $R^\ell = 0$ if $\ell \geq 1$, and henceforth we shall write $R_k$ for $R_k^0$. Moreover, we have that $R$ satisfies the following properties:

\emph{Duality principle}: A holomorphic section $\Phi$ of $E_0$ belongs to $\im f_1$ if and only if $R \Phi = 0$.

\emph{Dimension principle}: If $\codim Z > k$, then $R_k = 0$.

Note that the second equality in (\ref{eq:nabla-U-R}) is equivalent to
\begin{align}
	f_1 R_1 &= 0, \label{eq:nabla-R-1} \\
	f_{k+1} R_{k+1} - \dbar R_k &= 0, \quad 1 \leq k \leq n-1, \label{eq:nabla-R-2} \\
	\dbar R_n &= 0. \label{eq:nabla-R-3}
\end{align}

Let $i: X \hookrightarrow Y$ be a closed complex subspace with ideal sheaf $\isheaf_X$, and suppose that $\holo_X = i^* \holo_Y$, which we identify with $\holo_Y / \isheaf_X$, has a Hermitian resolution of the form
\begin{equation}
\label{eq:locally-free-res}
	0 \longrightarrow
	\holo(E_n) \overset{f_n}{\longrightarrow} \dots
	\overset{f_2}{\longrightarrow} \holo(E_1) \overset{f_1}{\longrightarrow} \holo_Y
	\longrightarrow \holo_X \longrightarrow 0,
\end{equation}
cf.\ (\ref{eq:hermitian-res}) where $E_0$ is the trivial line bundle. For the associated residue current $R = R_1 + \dots + R_n$, we can view each $R_k$ as an $E_k$-valued $(0,k)$-current. Since $\im f_1 = \isheaf_X$, we have that $i^* \Phi = 0$ if and only if $R \Phi = 0$ by the duality principle. More generally, let $L \rightarrow Y$ be a holomorphic line bundle. If we equip $L$ with a Hermitian metric, then we obtain a Hermitian resolution of $i^* L = \holo_X \otimes L$ by tensoring (\ref{eq:locally-free-res}) with $L$, and we have that $R$ is the associated residue current with this resolution as well.
\subsection{The Coleff--Herrera product}
\label{section:ch}
Let $f = (f_1, \dots, f_p): \C^n \rightarrow \C^p$ be a holomorphic mapping such that $V(f) = f^{-1}(0)$ has codimension $p$. In \cite{ch} Coleff and Herrera gave meaning to the product
\begin{equation}
\label{eq:ch-product}
	\mu^f =
	\dbar \left[ \frac{1}{f_1} \right] \wedge \dots \wedge
	\dbar \left[ \frac{1}{f_p} \right],
\end{equation}
which is known as the \emph{Coleff--Herrera product}. In particular, if each $f_j$ only depends on $\zeta_j$, then (\ref{eq:ch-product}) is just the tensor product of the one-variable currents $\dbar [1/f_j]$ described above. The current $\mu^f$ is $\dbar$-closed, has support $V(f)$, and is anti-commuting in the $f_j$. Moreover, $\mu^f$ satisfies the duality principle, i.e., $\mu^f \Phi= 0$ if and only if $\Phi \in \isheaf(f)$, where $\isheaf(f)$ is the ideal sheaf generated by $f$.

Let $H \rightarrow Y$ be a holomorphic Hermitian vector bundle of rank $p$, and let $f$ be a holomorphic section of the dual bundle $H^*$. Let $E_k = \bigwedge^k H$, and define $\delta_k: E_k \rightarrow E_{k-1}$ as interior multiplication by $f$. This gives a generically exact complex (\ref{eq:bundle-complex}). Suppose $f = f_1 e_1^* + \dots + f_p e_p^*$ in some local holomorphic frame $e_j^*$ for $H^*$. If $\codim f^{-1}(0) = p$, then the corresponding complex of sheaves is a Hermitian resolution of $\holo_Y/\isheaf(f)$ known as the \emph{Koszul complex}, and it was proven in \cite{mats} that the associated residue current is given by $R = R_p = \mu^f e_1 \wedge \dots \wedge e_p$.
\subsection{A comparison formula for residue currents}
\label{section:comparison-formula}
We have the following \emph{comparison formula} for residue currents, see Theorem~1.3 and Corollary~4.7 in \cite{larkang}. Let $X \subset X'$ be complex subspaces of codimension $p$ of $Y$. Suppose that there exist Hermitian resolutions of length $p$ of $\holo_X$ and $\holo_{X'}$, respectively, and let $R$ and $R'$ be the associated residue currents. Moreover, suppose that there exists a map of complexes
\[
	\begin{tikzcd}
	0 \ar{r} &[-10pt] \holo(E'_p) \ar{d}{\psi_p} \ar{r}{f'_p} & \cdots \ar{r}{f'_2} &
	\holo(E'_1) \ar{d}{\psi_1} \ar{r}{f'_1} & \holo_Y \ar{d}{\id}
	\ar{r} & \holo_{X'} \ar{d} \ar{r} &[-10pt] 0 \\
	0 \ar{r} & \holo(E_p) \ar{r}{f_p} & \cdots \ar{r}{f_2} & \holo(E_1)
	\ar{r}{f_1} & \holo_Y \ar{r} & \holo_X \ar{r} & 0
	\end{tikzcd}
\]
Then $R_p = \psi_p R'_p$.
\section{Interpolation and residue currents}
\label{section:interpolation}
Let $Y$ be a complex manifold of dimension $n$, and let $i: X \hookrightarrow Y$ be a closed complex subspace. Let $L \rightarrow Y$ be a holomorphic line bundle, and let $\Phi$ and $\varphi$ be global holomorphic sections of $L$ and $i^* L$, respectively. We say that $\Phi$ interpolates $\varphi$ if $i^* \Phi = \varphi$.

Suppose that there exists a Hermitian resolution of $\holo_X$ of the form (\ref{eq:locally-free-res}), and let $R$ denote the associated residue current. For each point $x \in Y$ there is a neighborhood $\opens$ and a holomorphic section $\widetilde{\varphi}$ of $L$ such that $i^* \widetilde{\varphi} = \varphi$ on $\opens$. We define the current $R \varphi$ on $Y$ locally as $R \widetilde{\varphi}$. This is well-defined since if $\widetilde{\varphi}'$ is another section such that $i^* \widetilde{\varphi}' = \varphi$, then $R (\widetilde{\varphi} - \widetilde{\varphi}') = 0$ by the duality principle.

We have the following result which follows immediately as a special case of Lemma 4.5 (ii) in \cite{aw2}.
\begin{lemma}
\label{lemma:nabla-exact}
Let $\Phi$ and $\varphi$ be global holomorphic sections of $L$ and $i^* L$, respectively. Then $\Phi$ interpolates $\varphi$ if and only if there exists a current $w$ such that $\Phi - R \varphi = \nabla w$.
\end{lemma}
In other words, $\varphi$ has an interpolant if and only if there exist currents $w_1, \dots, w_n$ such that $\dbar w_n = R_n \varphi$, and
\begin{equation}
\label{eq:dbar-xi}
	\dbar w_k = f_{k+1} w_{k+1} + R_k \varphi,
	\quad 1 \leq k \leq n-1.
\end{equation}
Moreover, in this case an interpolant of $\varphi$ is given by $\Phi = f_1 w_1$. Note that $\Phi$ is holomorphic since $\dbar \Phi = -f_1 \dbar w_1 = -f_1(f_2 w_2 + R_1 \varphi) = -(f_1 R_1) \varphi = 0$. Here the last equality follows from (\ref{eq:nabla-R-1}).

Let us now consider interpolation on $Y = \P^n$ with respect to the line bundle $L = \holo(d)$. Recall that there is a Hermitian resolution of $\holo_X$ of the form (\ref{eq:locally-free-res-pn}). We write $R_{k,\ell}$ for the $\holo(-\ell)^{\beta_{k,\ell}}$-valued component of $R_k$.

If $R_n \varphi = \dbar w_n$ for some current $w_n$, then one can successively find currents $w_{n-1}, \dots, w_1$ such that (\ref{eq:dbar-xi}) holds since, in view of (\ref{eq:nabla-R-2}),
\[
	\dbar(f_{k+1} w_{k+1} + R_k \varphi) = -f_{k+1} \dbar w_{k+1} + (\dbar R_k) \varphi =
	-(f_{k+1} R_{k+1} - \dbar R_k) \varphi = 0,
\]
and it follows from, e.g., \cite{demailly}*{Theorem 10.7} that
\[
	H^{0,k}(\P^n,E_k \otimes \holo(d)) = 0, \quad
	1 \leq k \leq n-1.
\]
We thus have the following condition for the existence of an interpolant.
\begin{lemma}
\label{lemma:residue-condition}
A global holomorphic section $\varphi$ of $\holo_X$ has an interpolant if and only if $R_n \varphi$ is $\dbar$-exact, i.e., there exists a current $\eta$ such that $R_n \varphi = \dbar \eta$.
\end{lemma}
\begin{proof}[Proof of Theorem~\ref{thm:moment-conditions}]
By Serre duality we have that $R_{n,\ell} \varphi$ is $\dbar$-exact if and only if
\[
	\int_{\P^n} R_{n,\ell} \varphi \wedge \eta = 0
\]
for all global $\dbar$-closed $\holo(\ell-d)$-valued $(n,0)$-forms $\eta$. Note that each such form is of the form $h \omega$ for some global holomorphic section $h$ of $\holo(\ell-d-n-1)$. Since $R_n \varphi$ is $\dbar$-exact if and only if each component $R_{n,\ell} \varphi$ is, the statement follows from Lemma~\ref{lemma:residue-condition}.
\end{proof}
\begin{proof}[Proof of Corollary~\ref{cor:interpolation-degree}]
Let $d \geq t_n(S_X) - n$, see (\ref{eq:tk-sx}), and let $\varphi$ be a global holomorphic section of $\holo_X(d)$. We have for each $\ell$ that
\[
	\int_{\P^n} R_{n,\ell} \varphi \wedge h \omega = 0
\]	
for all global holomorphic sections $h$ of $\holo(\ell-d-n-1)$. Indeed, if $\ell \leq d+n$, the only such $h$ is the zero section, and if $\ell > d+n$, then $R_{n,\ell} = 0$ since $\beta_{n,\ell} = 0$. Therefore $\varphi$ has an interpolant by Theorem~\ref{thm:moment-conditions}.
\end{proof}
\section{Polynomial interpolation}
\label{section:polynomial-interpolation}
Let us now return to the topic of polynomial interpolation. Recall that the setting is that $X$ is a complex subspace of $\C^n$ such that $X_{\red}$ is a finite set of points. The aim of this section is to give some examples where we explicitly compute the residue current $R$ associated with a Hermitian resolution of $\holo_X$ and write down the moment conditions that Theorem~\ref{thm:moment-conditions} imposes on a function $g \in A_X$ for the existence of an interpolant of degree at most $d$. We do this by identifying $g$ with a global holomorphic section $\varphi$ of $\holo_X(d)$ and use the fact that $g$ has an interpolant of degree at most $d$ if and only if $\varphi$ has an interpolant. More precisely, we let $[z] = [z_0 \,{:}\, \dots \,{:}\, z_n]$ denote homogeneous coordinates on $\P^n$, and we view $\C^n$ as an open complex subspace of $\P^n$ via the embedding $(\zeta_1,\dots,\zeta_n) \mapsto [1 \,{:}\, \zeta_1 \,{:}\, \dots \,{:}\, \zeta_n]$. Recall that on $\C^n$ there is a frame $e$ for $\holo(1)$ such that a global holomorphic section $\Phi$ of $\holo(d)$ is given by
\[
	\Phi(\zeta_1,\dots,\zeta_n) =
	G(\zeta_1,\dots,\zeta_n) e(\zeta_1,\dots,\zeta_n)^{\otimes d},
\]
where $G$ is a polynomial of degree at most $d$ on $\C^n$.

Throughout this section we shall let $\omega$ in Theorem~\ref{thm:moment-conditions} be the nonvanishing holomorphic $\holo(n+1)$-valued $n$-form on $\P^n$ such that
\[
	\omega =
	d \zeta_1 \wedge \dots \wedge d \zeta_n  \otimes e^{\otimes (n+1)}
\]
on $\C^n \subset \P^n$.

Note that the dimension principle gives that $R = R_n$, and throughout this section we write $R_\ell$ rather than $R_{n,\ell}$ for the $\holo(-\ell)^{\beta_{n,\ell}}$-valued component of $R_n$.
\begin{example}
Let $X = \{ (0,0), (1,0), (0,1), (1,1) \} \subset \C^2 \subset \P^2$. We have that $X$ is defined by the homogeneous ideal $I_X = (f_1,f_2)$, where $f_1 = z_1(z_1-z_0)$ and $f_2 = z_2(z_2-z_0)$. A Hermitian resolution of $\holo_X$ is given by the Koszul complex, see Section~\ref{section:ch}, where we interpret $(f_1,f_2)$ as a global holomorphic section of $\holo(2)^2$. Thus the associated residue current takes values in $\holo(-4)$, and is given by the Coleff--Herrera product, see Section~\ref{section:ch},
\[
	R = R_4 =
	\dbar \left[ \frac{1}{\zeta_1(\zeta_1-1)} \right] \wedge
	\dbar \left[ \frac{1}{\zeta_2(\zeta_2-1)} \right] e^{\otimes (-4)}.
\]
By a straightforward computation, cf.\ (\ref{eq:principal-value}), we get
\begin{align*}
	R_4 &=
	\left(
	\dbar \left[ \frac{1}{\zeta_1} \right] \wedge \dbar \left[ \frac{1}{\zeta_2} \right] -
	\dbar \left[ \frac{1}{\zeta_1 - 1} \right] \wedge \dbar \left[ \frac{1}{\zeta_2} \right]
	\right.
	\\ &-
	\left.
	\dbar \left[ \frac{1}{\zeta_1} \right] \wedge \dbar \left[ \frac{1}{\zeta_2 - 1} \right] +
	\dbar \left[ \frac{1}{\zeta_1 - 1} \right] \wedge \dbar \left[ \frac{1}{\zeta_2 - 1} \right]
	\right) e^{\otimes (-4)}.
\end{align*}
By Theorem~\ref{thm:moment-conditions}, we now get the following. Since $R_\ell = 0$ for $\ell \geq 5$, we have that any $g \in A_X$ has an interpolant of degree at most $2$. Moreover, $g$ has an interpolant of degree at most 1 if and only if (\ref{eq:integral-condition}) holds when $\ell = 4$ and $h = 1$. In view of (\ref{eq:dbar-zk}) this amounts to
\begin{equation}
\label{eq:ex1-eq}
	g(0,0) - g(1,0) - g(0,1) + g(1,1) = 0,
\end{equation}
which is expected since the values of $g$ at $(0,0)$, $(1,0)$, and $(0,1)$ uniquely determines a polynomial of degree at most 1 that takes the value $g(1,0) + g(0,1) - g(0,0)$ at $(1,1)$. Note that this gives that the interpolation degree of $X$ is 2.

We have that $g$ has a constant interpolant if and only if (\ref{eq:integral-condition}) holds for all global holomorphic $h$ of $\holo(1)$. By linearity we only need to check $h = z_0, z_1, z_2$, which amounts to (\ref{eq:ex1-eq}), $g(1,1) - g(1,0) = 0$, and $g(1,1) - g(0,1) = 0$. This amounts to
\[
	g(0,0) = g(1,0) = g(0,1) = g(1,1)
\]
as expected.
\qed
\end{example}
\begin{example}
Let $X = \{ (0,0), (1,0), (0,1), (0,2) \} \subset \C^2 \subset \P^2$. We have that $X$ is defined by the homogeneous ideal $I_X = (z_1 a_1, z_1 z_2, z_2 a_2)$, where $a_1 = z_1 - z_0$ and $a_2 = (z_2 - z_0)(z_2-2z_0)$. We have Hermitian resolutions of $\holo_{\P^2} / \isheaf(z_1 a_1, z_2 a_2)$ and $\holo_X$ and a map of complexes:
\[
	\begin{tikzcd}
	\holo(-5) \ar{d}{\psi_2}
	\ar{r}{\delta_2} & \holo(-2) \oplus \holo(-3) \ar{d}{\psi_1} \ar{r}{\delta_1}
	& \holo_{\P^2} \ar{d}{\id} \ar{r} & \holo_{\P^2} / \isheaf(z_1 a_1, z_2 a_2) \ar{d} \\
	\holo(-3) \oplus \holo(-4) \ar{r}{f_2} &
	\holo(-2)^2 \oplus \holo(-3) \ar{r}{f_1} & \holo_{\P^2} \ar{r} & \holo_X
	\end{tikzcd}
\]
where the upper complex is the Koszul complex, see Section~\ref{section:ch}. Moreover,
\[
	f_1 =
	\begin{bmatrix}
	z_1 a_1 & z_1 z_2 & z_2 a_2
	\end{bmatrix}, \quad
	f_2 =
	\begin{bmatrix}
	-z_2 & 0 \\
	a_1 & -a_2 \\
	0 & z_1
	\end{bmatrix},
\]
and
\[
	\psi_1 =
	\begin{bmatrix}
	1 & 0 \\
	0 & 0 \\
	0 & 1
	\end{bmatrix}, \quad
	\psi_2 =
	\begin{bmatrix}
	a_2 \\
	a_1
	\end{bmatrix}.
\]
Let $R$ and $R'$ denote the residue currents associated with the resolutions of $\holo_X$ and $\holo/\isheaf(z_1 a_1,z_2 a_2)$, respectively. We have that $R' = R'_5$ takes values in $\holo(-5)$ and is given by the Coleff--Herrera product, see Section~\ref{section:ch},
\[
	R'_5 =
	\dbar \left[ \frac{1}{\zeta_1(\zeta_1-1)} \right] \wedge
	\dbar \left[ \frac{1}{\zeta_2(\zeta_2-1)(\zeta_2-2)} \right] e^{\otimes (-5)}.
\]
Thus by the comparison formula, see Section~\ref{section:comparison-formula}, $R = \psi_2 R' = R_3 \oplus R_4$. A straightforward computation gives that
\begin{align*}
	R_3 &=
	\dbar \left[ \frac{1}{\zeta_1(\zeta_1-1)} \right] \wedge
	\dbar \left[ \frac{1}{\zeta_2} \right] e^{\otimes (-3)} \\ &=
	\left(
	-\dbar \left[ \frac{1}{\zeta_1} \right] \wedge \dbar \left[ \frac{1}{\zeta_2} \right] +
	\dbar \left[ \frac{1}{\zeta_1 - 1} \right] \wedge \dbar \left[ \frac{1}{\zeta_2} \right]
	\right) e^{\otimes (-3)},
\end{align*}
and
\begin{align*}
	R_4 &=
	\dbar \left[ \frac{1}{\zeta_1} \right]
	\wedge \dbar \left[ \frac{1}{\zeta_2(\zeta_2-1)(\zeta_2-2)} \right] e^{\otimes (-4)} \\ &=
	\left(
	\frac{1}{2} \dbar \left[ \frac{1}{\zeta_1} \right] \wedge
	\dbar \left[ \frac{1}{\zeta_2} \right] -
	\dbar \left[ \frac{1}{\zeta_1} \right] \wedge \dbar \left[ \frac{1}{\zeta_2 - 1} \right] +
	\frac{1}{2} \dbar \left[ \frac{1}{\zeta_1} \right] \wedge
	\dbar \left[ \frac{1}{\zeta_2 - 2} \right]
	\right) e^{\otimes (-4)}.
\end{align*}

By Theorem~\ref{thm:moment-conditions}, we now get the following. Since $R_\ell = 0$ for $\ell \geq 5$, we have that any $g \in A_X$ has an interpolant of degree at most $2$. Moreover, $g$ has an interpolant of degree at most 1 if and only if (\ref{eq:integral-condition}) holds when $\ell = 4$ and $h = 1$. (Note that there is no condition involving $R_3$ since $\ell - d - n -1 < 0$ in this case.) In view of (\ref{eq:dbar-zk}) we get the condition
\begin{equation}
\label{eq:ex2-eq}
	\frac{1}{2} g(0,0) - g(0,1) + \frac{1}{2} g(0,2) = 0,
\end{equation}
which is expected since $g$ has an interpolant of degree at most 1 if and only if $g(0,1)$ is the average of $g(0,0)$ and $g(0,2)$. Note that this gives that the interpolation degree of $X$ is 2.

We get that $g$ has a constant interpolant if and only if (\ref{eq:integral-condition}) holds when $\ell = 4$ for $h = z_0, z_1, z_2$, and when $\ell = 3$ and $h = 1$. Since $z_1$ vanishes on the support of $R_4$, this amounts to the equations (\ref{eq:ex2-eq}), $g(0,2) - g(0,1) = 0$ and $g(1,0) - g(0,0) = 0$. This amounts to
\[
	g(0,0) = g(1,0) = g(0,1) = g(0,2)
\]
as expected.
\qed
\end{example}
We end this note by considering Hermite interpolation. We refer to, e.g., \cites{calvi,spitz,traub}, and references therein for a classical survey of this topic.
\begin{example}
\label{ex:hermite}
Let $p_0, \dots, p_r \in \C$, and let $g$ be a holomorphic function on the complex subspace $X \subset \C$ defined by the ideal generated by $\prod_{j=0}^r (\zeta-p_j)$. Here we allow for the possibility that $p_i = p_j$ for some $i,j$, so that $X$ is nonreduced in general, and we denote the number of times that $p_j$ occurs by $m_j$. We have that a polynomial $G$ interpolates $g$ if and only if, for each $j = 0,\dots,r$,
\[
	G^{(k)}(p_j) = g^{(k)}(p_j), \quad
	k = 0,\dots,m_j-1.
\]

We say that a polynomial interpolates $g$ with respect to $p_0,\dots,p_k$, $k \leq r$, if it interpolates the pull-back of $g$ to the complex subspace defined by the ideal generated by $\prod_{j=0}^k (\zeta-p_j)$. We denote the unique polynomial that interpolates $g$ with respect to $p_0,\dots,p_k$ by $H[g;p_0,\dots,p_k]$. The coefficient of its $\zeta^k$-term is referred to as the $k$th \emph{divided difference} of $g$ and we denote it by $g[p_0,\dots,p_k]$. By induction it is not difficult to see that
\[
	H[g;p_0,\dots,p_r](\zeta) =
	\sum_{k=0}^r g[p_0,\dots,p_k]
	\prod_{j=0}^{k-1} (\zeta-p_j),
\]
see \cite{calvi}*{Theorem~1.8}. This is referred to as Newton's formula.

We claim that $g \in A_X$ has an interpolant of degree at most $d$ if and only if $(gh)[p_0,\dots,p_r] = 0$ for all polynomials $h$ of degree at most $r-d-1$. Let us show how this condition follows from Theorem~\ref{thm:moment-conditions}. Since the ideal is generated by a single element, we have that the associated residue current is given by
\[
	R_{r+1} =
	\dbar \left[ \frac{1}{\prod_{j=0}^r (\zeta-p_j)} \right] e^{\otimes (-r-1)}.
\]
Theorem~\ref{thm:moment-conditions} together with Stokes' formula gives that $g$ has an interpolant if and only if
\[
	\int_{\C}
	\dbar \left[ \frac{1}{\prod_{j=0}^r (\zeta-p_j)} \right] g h
	\wedge d \zeta =
	\int_{C_R}
	\frac{H[g h;p_0,\dots,p_r](\zeta)}{\prod_{j=0}^r (\zeta-p_j)} \, d \zeta = 0
\]
for all polynomials $h$ of degree at most $r-d-1$, where $C_R$ is a circle of radius $R \gg 0$. Here we have used the fact that $H[g h;p_0,\dots,p_r]$ interpolates $g h$. By letting $R \rightarrow \infty$, a direct calculation gives that the second integral is equal to $2 \pi i \cdot (g h)[p_0,\dots,p_r]$, and hence the claim follows.
\end{example}
\section*{Acknowledgments}
I would like to thank Elizabeth Wulcan and Mats Andersson for helpful discussions and comments on preliminary versions of this note.

\end{document}